\documentclass[12pt]{amsart}
\usepackage[utf8]{inputenc}
\usepackage{urwchancal} 
\usepackage[top=1in, left=1in, right=1in, bottom=1in]{geometry}
\usepackage{amsmath, amssymb, mathtools, graphicx}
\usepackage{amsthm}
\usepackage[dvipsnames]{xcolor}
\usepackage{epsfig,fancyhdr}
\usepackage{float}
\usepackage{setspace}
\usepackage{epstopdf}
\usepackage{import}
\usepackage{lineno}
\usepackage{stackrel}
\usepackage[allcolors = blue,colorlinks]{hyperref}
\usepackage[abs]{overpic}
\usepackage[center]{caption}
\usepackage{subcaption}
\usepackage{tikz-cd}
\usepackage{enumitem}
\counterwithin{table}{section}
\usepackage{contour}
\contourlength{.08em}
\usepackage{cite}
\usepackage{marginnote}
\usepackage{clipboard}
\usepackage{cleveref}

\numberwithin{equation}{section}
\theoremstyle{plain}
\newtheorem{theorem}[equation]{Theorem}
\newtheorem{lemma}[equation]{Lemma}
\newtheorem{proposition}[equation]{Proposition}
\newtheorem{corollary}[equation]{Corollary}

\newtheorem*{namedtheorem}{\theoremname}
\newcommand{\theoremname}{testing}
\newenvironment{named}[1]{\renewcommand{\theoremname}{#1}\begin{namedtheorem}}{\end{namedtheorem}}

\theoremstyle{definition}
\newtheorem{definition}[equation]{Definition}

\newtheorem{remark}[equation]{Remark}

\newcommand{\refthm}[1]{Theorem~\ref{Thm:#1}}
\newcommand{\reflem}[1]{Lemma~\ref{Lem:#1}}
\newcommand{\refprop}[1]{Proposition~\ref{Prop:#1}}
\newcommand{\refcor}[1]{Corollary~\ref{Cor:#1}}

\newcommand{\refsec}[1]{Section~\ref{Sec:#1}}
\newcommand{\reffig}[1]{Figure~\ref{Fig:#1}}

\newcommand{\HH}{{\mathbb{H}}}

\renewcommand{\thefootnote}{\ensuremath{\mathcal{\alph{footnote}}}}

\newcommand{\tri}{\boldsymbol{\mathcal{T}}}

\newcommand{\bdy}{\partial}

\title{Triangulations of the 3-sphere with knotted edge}

\author{Dionne Ibarra}
\thanks{School of Mathematics, Monash University, VIC 3800, Australia.}

\author{Daniel V. Mathews \textsuperscript{\thefootnote}}

\author{Jessica S. Purcell}

\author{Jonathan Spreer}

\begin{document}


\begin{abstract}
We prove that for any knot $K$, there exists a decomposition of the 3-sphere into tetrahedra such that a representative of $K$ is formed from one vertex and one edge of the decomposition. The proof is constructive, and based on fully augmented links. We use our method to produce a family of simplicial triangulations of the $3$-sphere with a loop of four edges representing a $t$-fold connected sum of the trefoil, whose number of tetrahedra grows linearly in $t$. This is best possible up to a constant factor.
\end{abstract}

\maketitle

\section{Introduction}
Combinatorial properties of triangulations of spheres have been an active field of study for more than a century. In 1922, Steinitz showed that every triangulation of the $2$-dimensional sphere can be realised as the boundary of a convex $3$-polytope~
\cite{Steinitz22Theorem}; see also~\cite{ziegler95}. Triangulations of ($d$-dimensional) spheres with this property are called {\em polytopal} and have a very straightforward combinatorial structure: They can be embedded into $(d+1)$-space with straight lines and faces such that their vertices are in convex position.

In dimension three, not all triangulations of the 3-sphere are polytopal~\cite{Barnette73}. Instead, simplicial triangulations of the $3$-dimensional sphere can be grouped into larger and larger classes of triangulations, putting fewer and fewer restrictions on their combinatorial properties. A (non-exhaustive) list of such classes, ordered by inclusion, is given in \cite[Theorem~2.1]{BenedettiZiegler11}.
See also~\cite{Bjoerner95TopMethods} and the introductions of~\cite{BenedettiZiegler11, Lutz04NonConstructible} for more details and a comprehensive overview of the field. 

One method to prove that a triangulation of a 3-sphere does not lie in a given class is to use knot theory. Namely, various results state that a loop of edges forming a sufficiently complicated knot inside a triangulation acts as an obstruction to being a triangulation of a certain type. As an example, Hachimori and Ziegler showed that a triangulation of $S^3$ is not ``constructible'' if there is any non-trivial knot in its 1-skeleton consisting of three edges~\cite{HachimoriZiegler}. 
Some of the first explicit examples of such triangulations were given by Lutz~\cite{Lutz04NonConstructible}, who constructed a simplicial 3-sphere with three distinguished edges forming a trefoil. This was generalised by Benedetti and Lutz~\cite{BenedettiLutz:KnotsInCollapsible}, who found explicit examples of triangulations of the 3-sphere with three edges forming the trefoil, the double trefoil, and the triple trefoil. We give more history below. 

Our main contribution is a construction to make this technique more powerful. 

\begin{named}{\refthm{1vertex}}
\Copy{Statement:1vertex}{For any knot $K$ in $S^3$, there is a one-vertex $\Delta$-complex for $S^3$ with an edge forming $K$.}
\end{named}

Here, a $\Delta$-complex is a generalisation of a (simplicial) triangulation. Our $\Delta$-complexes are obtained by successively gluing triangular faces in pairs until we obtain a closed complex. A $\Delta$-complex for $S^3$ is homeomorphic to $S^3$ under the quotient topology for the face pairings; see for example~\cite[p~103]{Hatcher:AT}.
In our case, the gluing process may cause all vertices to become identified to a single vertex, which is not allowed for simplicial triangulations.
In this paper, we will use the term \emph{triangulation} to refer to a simplicial complex; and \emph{one-vertex triangulation} for a non-simplicial, one-vertex $\Delta$-complex.

Note that $\Delta$-complexes are natural in the setting of 3-manifold topology, where they are often simply called ``triangulations;'' see for example~\cite{JacoRubinstein:0Efficient, Matveev}. 
Moreover, this setting is still relevant for simplicial triangulations: the second derived subdivision of any given one-vertex triangulation is simplicial. In this simplicial version, knotted edges now appear as edge loops of length four. Thus the following is an immediate consequence of \refthm{1vertex}.

\begin{corollary}\label{Cor:SimplicialLength4}
For any knot $K$ in $S^3$, there is a simplicial triangulation of the $3$-sphere with an edge loop of length four forming $K$.
\end{corollary}

Our result is constructive and explicit, given only a diagram of the knot, and we give an algorithm for the construction. The constructed $\Delta$-complexes are of reasonable size, arising from decompositions that in some cases are even conjectured to be minimal; see \refsec{twistknots}. 
We use our $\Delta$-complexes to give an upper bound on the number of tetrahedra required to decompose a $3$-sphere containing an edge forming a knot $K$, where the bound only depends on the input diagram of $K$; see \refthm{TetrahedronCount} and \refcor{hyperbolic_knot_tet_count}.

The techniques to prove \refthm{1vertex} arise from hyperbolic geometry, although the result is independent of any hyperbolic geometric properties of the knot complement. However, our method of proof does yield a $\Delta$-complex such that removing the single vertex and edge, and collapsing just one of the tetrahedra, yields an ideal triangulation of the complement of the knot $K$ in $S^3$. Here, an {\em ideal triangulation} is a decomposition of the knot complement into tetrahedra with their vertices removed (\emph{ideal tetrahedra}), such that under the quotient topology, the gluing yields a manifold homeomorphic to the complement of $K$ in $S^3$.

In an ideal triangulation, it is often useful to refer to the removed vertices (even though they no longer form part of the triangulation), and we call them \emph{ideal vertices}. 
When we need to distinguish between removed/ideal vertices and vertices which actually form part of a triangulation, we call the latter \emph{material} vertices.

Generically (in an appropriate sense), our ideal triangulations have very nice hyperbolic geometric properties. Namely, the complement of a generic knot $K$ in $S^3$ admits a hyperbolic structure, due to Thurston~\cite{Thurston}, and the given ideal triangulation is realised by positively oriented, convex ideal hyperbolic tetrahedra~\cite{HamPurcell}. 

Our result also has connections to computational geometry and topology. Recently, Burton and He used computational techniques to build a family of one-vertex triangulations of the 3-sphere with all their edges being non-trivial knots~\cite{BurtonHe23}. This disproved a conjecture on the nature of decompositions of Seifert fibre spaces. As Burton and He note, many questions in computational 3-manifold topology arise by observing properties of small decompositions, that is, 3-manifolds built using only a small number of tetrahedra. Our $\Delta$-complexes, although efficient with respect to the data associated to a knot diagram, lead to $\Delta$-complexes of the 3-sphere that can have arbitrarily large numbers of tetrahedra. We expect they may lead to additional counterexamples.

\subsection*{History of 3-sphere triangulations with knotted edges}\label{Sec:History}

Triangulations of $3$-spheres with knotted edge loops have been extensively studied in the setting of \emph{simplicial triangulations}, i.e., simplicial complexes with geometric carrier homeomorphic to the $3$-sphere. See ~\cite[Theorem~2.3]{BenedettiLutz:KnotsInCollapsible} for a summary of 
results by Benedetti, Ehrenborg, Hachimori, Shimokawa, and Ziegler~
\cite{BenedettiZiegler11,EhrenborgHachimori01,HachimoriShimokawa04,HachimoriZiegler}.

Benedetti and Lutz used one of these results to construct the first explicit example of 
a non-locally constructible (simplicial) triangulation of a $3$-sphere~\cite{BenedettiLutz:KnotsInCollapsible}.
Here, a triangulation of a sphere is \emph{locally constructible} if and only if it admits a Morse function with no critical $2$-face (triangle), in the language of Forman's discrete Morse theory~\cite{forman98-morse,forman02-morse}.
We can also ask for the class of $3$-sphere triangulations with all of its Morse functions having at least $1,2,3,\ldots$ critical 2-faces.

\begin{theorem}[Benedetti~\cite{Benedetti12}, Lickorish~\cite{Lickorish1991}]
\label{Thm:benedettilickorish}
Let $S$ be a (simplicial triangulation of a) 3-sphere with a subcomplex of $m$ edges, 
isotopic to the sum of $t$ trefoil knots. For any discrete Morse function $f$ on $S$
one has
\[ c_2 (f) \geq t-m+1, \]
where $c_2(f)$ denotes the number of critical 2-faces of $S$.
\end{theorem}

We use our construction from \refthm{1vertex} together with \refthm{benedettilickorish}
to deduce the following statement, proved in \refsec{trefoils}.

\begin{corollary}\label{Cor:DiscreteMorseCor}
For every $t \geq 3$, there exists a (simplicial) triangulation of the $3$-sphere $S_t$ such that every discrete Morse function $f$ on $S_t$ has at least $t-3$ critical $2$-faces and $S_t$ has at most $24^2 \cdot (48 \cdot t - 19)$ tetrahedra.
\end{corollary}

In \refsec{Bounds} we present explicit instructions on how to construct one-vertex triangulations $\tilde{S}_t$, with second derived subdivision resulting in $S_t$.

\subsection*{$H$-triangulations}

Our construction is inspired by the notion of {\em H-triangulations}, a tool from quantum topology.
Following Aribi, Gu\'eritaud, and Piguet-Nakazawa ~\cite[Section~2.1]{BA-G-PN:TwistKnots}, define an \emph{H-triangulation} to be a $\Delta$-complex
$\tri$ of a closed 3-manifold $M$ with a single vertex and a distinguished edge $E$; this edge represents a knot $K$ in $M$. We say $\tri$ is an \emph{H-triangulation} for $(M, K)$. 
For any twist knot $K$, Aribi, Gu\'eritaud, and Piguet-Nakazawa describe H-triangulations of $(S^3,K)$.
Thus \refthm{1vertex} generalises this result to say that for any knot $K$ in $S^3$, $(S^3,K)$ has an H-triangulation.
This triangulation is obtained from an ideal triangulation by cutting open along a suitable face, inserting a new tetrahedron, and putting back the vertices. This process is called a {\em T-surgery} in \cite{Piguet}, where Piguet-Nakazawa proved that T-surgery can be applied to any ideal triangulation of a link to obtain an H-triangulation of some closed manifold.

H-triangulations are used in \cite{BA-G-PN:TwistKnots} to study the quantum Teichm\"uller TQFT for twist knots, and to prove the Teichm\"uller TQFT volume conjecture for twist knots. 
The quantum Teichm\"uller TQFT was constructed by Andersen and Kashaev in~\cite{AndersenKashaev} where they describe H-triangulations. For the quantum applications, edges of the H-triangulation are given weights, related to angles of a geometric triangulation (i.e.~with convex hyperbolic ideal tetrahedra), and a conjectural limit is proposed that relates to Kashaev quantum dilogarithm invariants~\cite{Kashaev4}. These are specialisations of the coloured Jones polynomials~\cite{MurakamiMurakami}, at the origin of the hyperbolic volume conjecture~\cite{Kashaev6}.
The notion of H-triangulations is also considered by Kashaev, Luo, and Vartanov~\cite{KashaevLuoVartanov}, who use H-triangulations of the trefoil, figure-8, $5_2$ and $6_1$ knots to compute examples of quantum invariants.

We note that $\Delta$-complexes arising from our construction are frequently geometric, see work by Ham and Purcell~\cite{HamPurcell}, and with well understood volume bounds~\cite{FKP:Volume}. This property may turn out to be useful in potential future applications. 

\subsection*{Structure of this paper}
In \refsec{triangulations} we show how to obtain, from certain ideal triangulations of a knot complement $M-K$, a one-vertex triangulation of $M$ with an edge forming $K$.
In \refsec{mertriangle} we show how to construct these ideal triangulations of knot complements, which must contain faces spanning meridians. This construction uses fully augmented links, Dehn filling, and layered solid tori.
In \refsec{Bounds} we consider some applications, giving bounds on the number of tetrahedra in our triangulations.
Finally, in \refsec{Algorithm} we give a summary of our algorithm.

\section*{Acknowledgements}  
This project began during the MATRIX Research program \emph{Low Dimensional Topology: Invariants of Links, Homology Theories, and Complexity}. 
The authors acknowledge the support by the Australian Research Council. Ibarra, Mathews, and Purcell were supported by ARC DP210103136, Ibarra and Purcell by ARC DP240102350, and Spreer by ARC DP220102588.

\section{From ideal to one-vertex triangulations}\label{Sec:triangulations}

Let $M$ be a closed orientable 3-manifold, and let $K$ be a knot in $M$, with $N(K)$ a regular neighbourhood. Suppose $M-K$ has an ideal triangulation $\tri$, or equivalently, $M-N(K)$ has a cell decomposition into truncated tetrahedra, each with all their truncated vertices on $\bdy N(K)$. Hence, the edges of $\tri$ have both endpoints on ideal vertices, and are thus called \emph{ideal edges}.
Suppose the ideal triangulation has a distinguished face $T$ that \emph{spans a meridian} of $N(K)$: that is, two of the ideal edges $e_1$ and $e_2$ of $T$, meeting at the ideal vertex $v$ of $T$ and oriented towards $v$, are identified to each other in an orientation preserving manner, and in a small neighbourhood of $T$ near $v$, the face $T$ intersects $\bdy N(K)$ in a closed curve forming a meridian. We say that $T$ forms a \emph{hat triangle}, following {\em Regina} \cite{regina} nomenclature.\footnote{Such a $T$ is also known as a \emph{horn triangle}.}
We assume throughout this article that $T$ is adjacent to two distinct tetrahedra.

In order to obtain an H-triangulation from $\tri$, we consider the following construction: Cut open the ideal triangulation along the face $T$. Label the two resulting boundary faces $s$ and $s'$. Insert a folded tetrahedron with face pairings given in \reffig{Tetrahedron} to obtain a closed H-triangulation. 

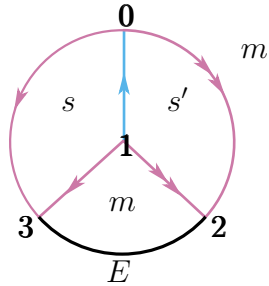
\begin{figure}[ht]
  \centering
$\vcenter{\hbox{\begin{overpic}[scale=1]{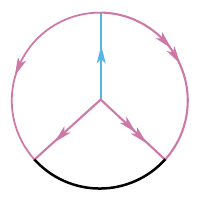}
    \put(25, 60){$s$}
\put(93, 80){$m$}
\put(42, -5){$E$}
\put(43, 22){$m$}
\put(65, 60){$s'$}
\put(47, 43){$\pmb 1$}
\put(47, 92){$\pmb 0$}
\put(82, 12){$\pmb 2$}
\put(9, 12){$\pmb 3$}
\end{overpic}}}$

\caption{The tetrahedron $\pmb \Delta$ to be inserted has faces $s$ and $s'$ glued to the cut-open face $T$ that was spanning the meridian. The faces labelled $m$ are folded over the edge $E$ illustrated in black, which becomes the knot.}
\label{Fig:Tetrahedron}
\end{figure}

\begin{theorem}\label{Thm:IdealToFiniteTriangulation}
Let $M$ be a closed orientable 3-manifold, $K$ a knot in $M$, and $\tri$ an ideal triangulation of $M-K$. Moreover, let $T$ be a face of $\tri$ that spans a meridian of $N(K)$. Then inserting the folded tetrahedron $\pmb \Delta$ along $T$ as shown in \reffig{Tetrahedron} yields a one-vertex 
triangulation for $M$, where the vertex is \emph{material}, i.e.\ not ideal. The distinguished edge $E$ in $\pmb \Delta$ forms the knot $K \subset M$.
\end{theorem}

The main steps of the proof of \refthm{IdealToFiniteTriangulation} are carried out in \reflem{CuspNbhdSphere}, \reflem{NbhdEdge}, and \reflem{DehnFillMeridian}. We start with the following immediate consequence:

\begin{corollary}\label{Cor:KnottedEdgeS3}
When $M=S^3$ in \refthm{IdealToFiniteTriangulation}, $E$ forms the knot $K$ in the resulting one-vertex triangulation of $S^3$.
\end{corollary}

\begin{lemma}\label{Lem:CuspNbhdSphere}
Gluing the tetrahedron $\pmb{\Delta}$ from \reffig{Tetrahedron} into an ideal triangulation with one ideal vertex along a hat triangle yields a one-vertex triangulation of a closed manifold. That is, the neighbourhood of the vertex class is homeomorphic to the $3$-ball. 
\end{lemma}

\begin{proof}
Figures~\ref{Fig:CuspNeighbourhoodCut}, \ref{Fig:CuspNeighbourhoodGlue}, \ref{Fig:Final} show that the boundary of a neighbourhood of the ideal vertex becomes a sphere. Thus when the neighbourhood pieces are patched together, they form a ball with this sphere as boundary.

In more detail, the torus $\bdy N(K)$ inherits a \emph{cusp triangulation} from the ideal triangulation of $M-K$: this is the intersection of $\bdy N(K)$ with the ideal tetrahedra. Thus the triangles of $\bdy N(K)$ come from triangles parallel to truncated ideal vertices. The distinguished face $T$ is adjacent to two tetrahedra, which we call $\Delta_1$ and $\Delta_2$, and we label the ideal vertices of $\Delta_j$, $j = 1,2$, by $0$, $1$, $2$, and $3$ such that face $T$ is identified with face $(013)$ of $\Delta_1$ and face $(310)$ of $\Delta_2$. The face $T$ has three ideal vertices, each of which meets the cusp triangulation; in the cusp triangulation we see three edges of triangles corresponding to intersections of $T$ with $\bdy N(K)$. In Figure~\ref{Fig:CuspNeighbourhoodCut}(A), these edges are drawn in orange.
By hypothesis, one of them is an edge spanning a meridian, which we have placed on the left and right of a quadrilateral fundamental domain for $\bdy N(K)$; this is the thick vertical edge, drawn on both vertical sides of the rectangle. The other two edges coming from intersections of $T$ with $\bdy N(K)$ share a common endpoint (as $T$ is a hat triangle), namely the orange point drawn in the middle of the horizontal sides.

We label the ideal vertices of both $\Delta_1$ and $\Delta_2$ such that the meridian coming from $T$ runs across the triangles at the ideal vertex labelled $1$. Hence, the cusp triangles on the far left and far right are parallel to the ideal vertices $\Delta_1(1)$ and $\Delta_2(1)$ of tetrahedra $\Delta_1$ and $\Delta_2$, respectively. The four edges running from $1$ down to $0$ and from $1$ down to $3$ in $\Delta_j$, $j = 1,2$ are all glued together. Call the resulting edge class $e$. The intersection of $\bdy N(K)$ with one end of this edge class $e$ forms the corners of the quadrilateral.
We have chosen the quadrilateral fundamental domain so that the upper and lower sides of the quadrilateral also meet $e$ at its other end, as it runs from vertex $0$ down to $1$ and from $3$ down to $1$ in each of $\Delta_1$, $\Delta_2$. Adjacent to these, we see triangles labelled $\Delta_j(0)$ and $\Delta_j(3)$, $j=1,2$. 

Now slice open the ideal triangulation of $M-K$ along the face $T$. The cusp triangulation changes as in \reffig{CuspNeighbourhoodCut}(B). Slicing open yields two unglued ideal triangles, whose intersections with $\bdy N(K)$ are shown in \reffig{CuspNeighbourhoodCut}(B) in orange and green. 

\begin{figure}[ht]
 $$\vcenter{\hbox{\begin{overpic}[scale=1]{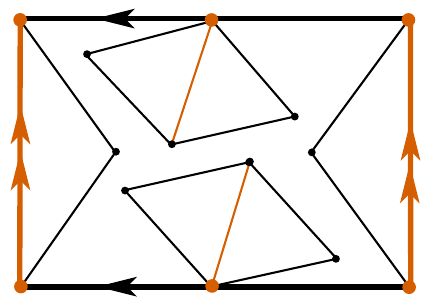}
    \put(14, 82){$\Delta_1(1)$}
\put(12, 118){$0$}
\put(12, 22){$3$}
\put(45, 71){$2$}
 \put(165, 83){$\Delta_2(1)$}
\put(188, 116){$3$}
\put(188, 24){$0$}
\put(158, 71){$2$}
\put(62, 111){$\Delta_i(0)$}
\put(90, 124){$1$}
\put(52, 113){$2$}
\put(78, 85){$3$}
\put(95, 104){$\Delta_j(3)$}
\put(102, 123){$1$}
\put(91, 84){$0$}
\put(128, 90){$2$}
\put(82, 43){$\Delta_i(3)$}
\put(109, 57){$0$}
\put(71, 47){$2$}
\put(97, 18){$1$}
\put(115, 30){$\Delta_j(0)$}
\put(119, 53){$3$}
\put(110, 16){$1$}
\put(148, 23){$2$}
\put(95, -7){(A)}
\end{overpic}}} \qquad
 \vcenter{\hbox{\begin{overpic}[scale=1]{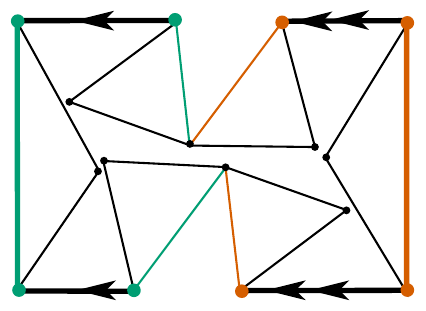}
    \put(10, 75){$\Delta_1(1)$}
\put(12, 110){$0$}
\put(12, 26){$3$}
\put(37, 62){$2$}
 \put(165, 80){$\Delta_2(1)$}
\put(187, 110){$3$}
\put(187, 27){$0$}
\put(163, 69){$2$}
\put(58, 108){$\Delta_1(0)$}
\put(78, 124){$1$}
\put(47, 99){$2$}
\put(81, 88){$3$}
\put(114, 100){$\Delta_2(3)$}
\put(131, 124){$1$}
\put(104, 82){$0$}
\put(141, 82){$2$}
\put(61, 45){$\Delta_1(3)$}
\put(93, 58){$0$}
\put(56, 59){$2$}
\put(65, 20){$1$}
\put(117, 36){$\Delta_2(0)$}
\put(114, 55){$3$}
\put(118, 19){$1$}
\put(150, 42){$2$}
\put(95, -7){(B)}
\end{overpic}}}$$
 \caption{(A) Start with cusp neighbourhood where $i=1$ and $j=2$ or $i=2$ and $j=1$. (B) Cut along face $T$ to unglue face $\Delta_1(013)$ from $\Delta_2 (013)$.} \label{Fig:CuspNeighbourhoodCut}
\end{figure}
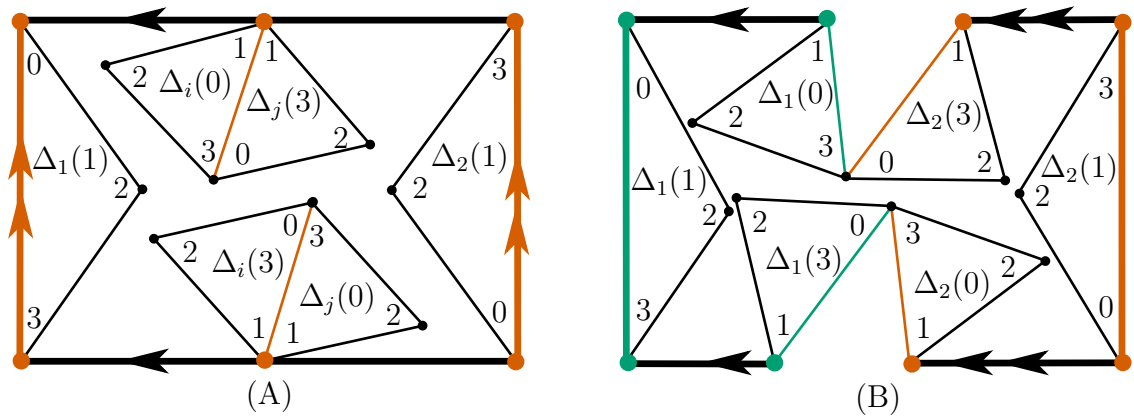

When we attach the new tetrahedron $\pmb \Delta$ of \reffig{Tetrahedron}, we add a new vertex to the cusp triangulation, corresponding to the edge $E$, as well as
four new triangles (parallel to the vertices of $\pmb \Delta$) and their side gluings.
These four triangles with gluings coming from faces $s$ and $s'$ are shown in \reffig{CuspNeighbourhoodGlue}. Here $s$ is the orange face, and $s'$ is the green face. We have not yet identified the sides of the triangles coming from identifying the faces labelled $m$. However, these sides of triangles in the figure are each labelled with an arrow and one, two, or three tick marks, indicating how they must be identified. 

\begin{figure}[ht]
  \centering
$\vcenter{\hbox{\begin{overpic}[scale=1]{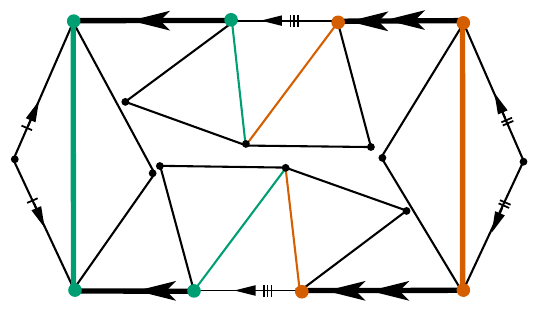}
  \put(37, 74){$\Delta_1(1)$}
\put(40, 108){$0$}
\put(40, 28){$3$}
\put(64, 61){$2$}
 \put(192, 80){$\Delta_2(1)$}
\put(213, 110){$3$}
\put(213, 27){$0$}
\put(188, 69){$2$}
\put(86, 108){$\Delta_i(0)$}
\put(105, 124){$1$}
\put(74, 99){$2$}
\put(108, 88){$3$}
\put(141, 100){$\Delta_j(3)$}
\put(158, 124){$1$}
\put(131, 82){$0$}
\put(168, 82){$2$}
\put(89, 45){$\Delta_i(3)$}
\put(121, 58){$0$}
\put(83, 59){$2$}
\put(93, 20){$1$}
\put(145, 36){$\Delta_j(0)$}
\put(142, 55){$3$}
\put(146, 19){$1$}
\put(178, 42){$2$}
\put(20, 70){$\pmb \Delta$}
\put(125, 30){$\pmb \Delta$}
\put(125, 110){$\pmb \Delta$}
\put(230, 70){$\pmb \Delta$}
\end{overpic}}} $
  \caption{Glue faces $\Delta_1(013)$ and $\Delta_2(013)$ to faces $s$ (orange) and $s'$ (green) of $\Delta$. The new vertices at left and right are ends of $E$. 
	} \label{Fig:CuspNeighbourhoodGlue}
\end{figure}
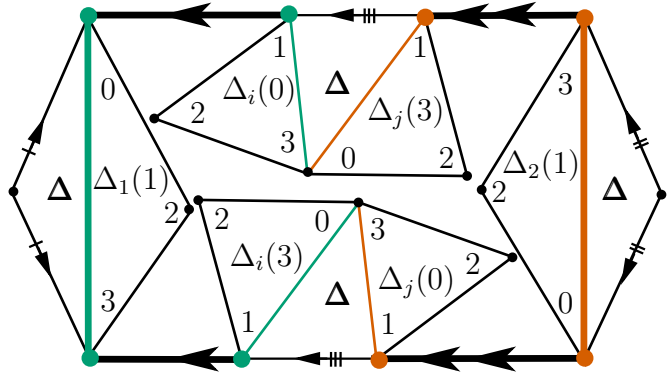

Identifying faces labelled $m$ folds the two triangles on the far left and far right, identifying their top and bottom, as indicated by arrows with single and double ticks in \reffig{CuspNeighbourhoodGlue}. We then identify the double arrowed edges and the single arrowed edges. The result is a disc with two sides, coming from faces labelled $m$, shown in \reffig{Final}. 
\begin{figure}[ht]
   \centering
$\vcenter{\hbox{\begin{overpic}[scale=1]{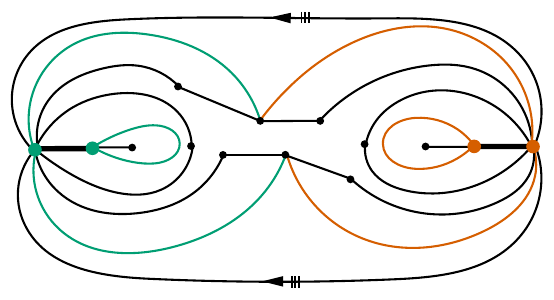}
    \put(80, 109){$\Delta_1$}
\put(160, 105){$\Delta_2$}
\put(170, 35){$\Delta_2$}
 \put(232, 78){$\Delta_2$}
\put(82, 34){$\Delta_1$}
\put(27, 78){$\Delta_1$}
\put(130, 16){$\pmb \Delta$}
\put(130, 110){$\pmb \Delta$}
 \put(190, 70){$\pmb \Delta$}
 \put(68, 69){$\pmb \Delta$}
\end{overpic}}} $
  \caption{Result.}
  \label{Fig:Final}
\end{figure}
These two sides are identified, creating a sphere.
\end{proof}

\begin{lemma}\label{Lem:NbhdEdge}
Let $M'$ be a closed manifold with a one-vertex triangulation $\mathcal{T}$ obtained from $M-K$ with distinguished edge $E$ described in \reflem{CuspNbhdSphere}. Then the complement of a tubular neighbourhood of $E$ is a manifold homeomorphic to $M-N(K)$. 
\end{lemma}

\begin{proof}
Note that endpoints of $E$ lie on the single vertex of $\mathcal{T}$, so a tubular neighbourhood of $E$ contains this vertex. Start by removing a neighbourhood of the vertex from the closed manifold $M'$.
This truncates all vertices of all tetrahedra; the truncated distinguished tetrahedron $\pmb \Delta$ is shown in \reffig{FoldedTetEremoved}(A).

Now remove a neighbourhood $N(E)$ of the distinguished edge $E$. The effect on the tetrahedron $\pmb \Delta$ is shown in \reffig{FoldedTetEremoved}(B), and again in \reffig{FoldedTetEremoved}(C), where (C) shows the tetrahedron as viewed from the cusp neighbourhood of the knot complement. The remnants of faces $m$ are shaded in these figures.

Note that the truncated tetrahedron now forms a regular neighbourhood of the original face $T$, as in \reffig{FoldedTetEremoved}(D). This deformation retracts to $T$, giving a homeomorphism to $M-N(K)$.
\end{proof}

\begin{figure}[ht]
  \input{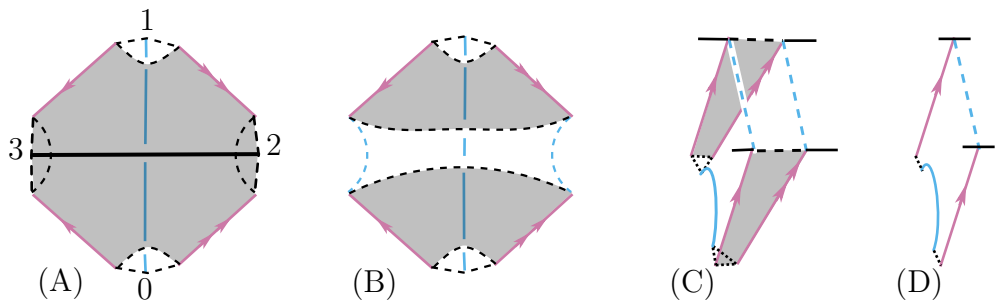}
  \caption{An illustration of the proof of Lemma~\ref{Lem:NbhdEdge}. Dotted blue lines bound a disc in $N(E)$. In (B,C), the pairs of dotted blue lines are identified.}
  \label{Fig:FoldedTetEremoved}
\end{figure}

\begin{lemma}\label{Lem:DehnFillMeridian}
Inserting tetrahedron $\pmb \Delta$ performs Dehn filling on $M-N(K)$, where the slope of the Dehn filling corresponds to the slope spanned by the triangular face $T$. Thus the Dehn filling is along a meridian, and yields $M$. 
\end{lemma}

\begin{proof}
By \reflem{NbhdEdge}, the manifold obtained by adding the tetrahedron $\pmb \Delta$ of \reffig{Tetrahedron} is obtained from $M-N(K)$ by inserting a solid torus equal to a neighbourhood of the edge $E$. Inserting a solid torus into $M-N(K)$ is a Dehn filling. It remains to check that the Dehn filling slope is the meridian, to yield $M$. Note that the Dehn filling slope is the unique non-trivial slope on $\bdy N(K)$ that bounds a disc after Dehn filling. 

A meridian of $M-N(K)$ is shown as the thick orange vertical edge on the left and right of the fundamental region in \reffig{CuspNeighbourhoodCut}(A). Gluing in $\pmb \Delta$ with $N(E)$ removed inserts this truncated tetrahedron as a neighbourhood of the face $T$, with the boundary of a disc shown as the blue dotted line in \reffig{FoldedTetEremoved}(B), (C), and (D). 
Note this is exactly the slope spanned by $T$.
\end{proof}

\begin{proof}[Proof of \refthm{IdealToFiniteTriangulation}]
By \reflem{DehnFillMeridian}, inserting $\pmb \Delta$ performs meridional Dehn filling on the manifold $M-K$, so the result is $M$. Moreover, it yields a $\Delta$-complex for $M$ with just one vertex, which is shown in \reflem{CuspNbhdSphere}. 

By \reflem{NbhdEdge}, the complement of a regular neighbourhood of the distinguished edge is homeomorphic to $M-N(K)$. A regular neighbourhood of the union of the edge and the single vertex forms a solid torus in $M$. Since the meridian of the solid torus is mapped to the meridian of $M-N(K)$, the homeomorphism from $M-N(E)$ to $M-N(K)$ extends to a homeomorphism of $M$, taking $E$ to $K$. Therefore, the edge forms the knot $K$.
\end{proof}

\section{Knot complements with a meridional triangle}
\label{Sec:mertriangle}

We first apply \refthm{IdealToFiniteTriangulation} to an infinite family of (hyperbolic) knot complements with known ideal triangulations, before utilising the machinery of hyperbolic fully augmented links to apply the theorem to complements of arbitrary knots. 

\begin{proposition}\label{Prop:2Bridge}
Let $K$ be a hyperbolic $2$-bridge knot (i.e.\ not a $(2,q)$-torus knot). There exists a one-vertex triangulation of $S^3$ with an edge forming $K$.
\end{proposition}

\begin{proof}
The well-known Sakuma--Weeks ideal triangulations of hyperbolic $2$-bridge knots, studied in \cite{SakumaWeeks} and shown to be geometric in~\cite{GueritaudFuter:2Bridge}, always contain a face spanning a meridian. The tetrahedron containing this face lies between the two so-called ``hairpin turns''. See, for example~\cite[Figure~19]{GueritaudFuter:2Bridge}, or~\cite[Figure~10.12]{Purcell:HypKnotTheory}. \refthm{IdealToFiniteTriangulation} then proves the result. 
\end{proof}

We generalise \refprop{2Bridge} to any knot in the 3-sphere by relating knots to Dehn fillings of {\em hyperbolic fully augmented links} --- a well-studied tool in hyperbolic geometry.

In the remainder of this section, we first recall the definition of a fully augmented link. We then show that every knot complement in the 3-sphere is obtained by Dehn filling a hyperbolic fully augmented link. Finally, we prove that a hyperbolic fully augmented link and its Dehn fillings can be given an ideal triangulation to ensure a face spans a meridian. The construction is explicit at every stage, allowing us to bound the number of tetrahedra. 

\begin{definition}[Fully augmented link]\label{Def:FAL}
A \emph{flat fully augmented link} in $S^3$ is a link with a diagram consisting of two types of components: \emph{Knot strands} are closed curves embedded in the plane of projection; \emph{crossing circles} are simple unknots meeting the plane of projection orthogonally, bounding a disc that is punctured by the knot strands exactly twice. A \emph{fully augmented link} is obtained from a flat fully augmented link by inserting one or zero crossings into the knot strands in a neighbourhood of a crossing circle.
\end{definition}

An example of a flat fully augmented link is shown in \reffig{FALExample}(A), and a more general fully augmented link in \reffig{FALExample}(D).

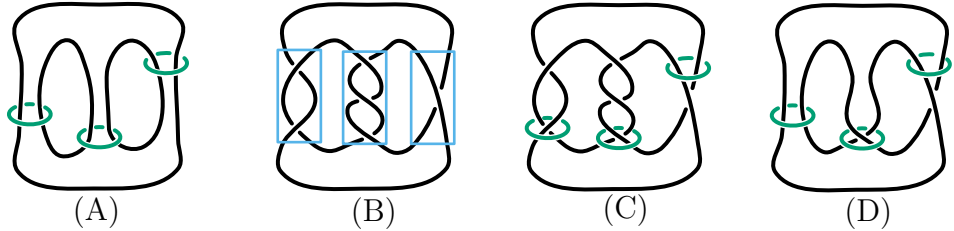
\begin{figure}
\begingroup%
  \makeatletter%
  \providecommand\color[2][]{%
    \errmessage{(Inkscape) Color is used for the text in Inkscape, but the package 'color.sty' is not loaded}%
    \renewcommand\color[2][]{}%
  }%
  \providecommand\transparent[1]{%
    \errmessage{(Inkscape) Transparency is used (non-zero) for the text in Inkscape, but the package 'transparent.sty' is not loaded}%
    \renewcommand\transparent[1]{}%
  }%
  \providecommand\rotatebox[2]{#2}%
  \newcommand*\fsize{\dimexpr\f@size pt\relax}%
  \newcommand*\lineheight[1]{\fontsize{\fsize}{#1\fsize}\selectfont}%
  \ifx\svgwidth\undefined%
    \setlength{\unitlength}{356.53951263bp}%
    \ifx\svgscale\undefined%
      \relax%
    \else%
      \setlength{\unitlength}{\unitlength * \real{\svgscale}}%
    \fi%
  \else%
    \setlength{\unitlength}{\svgwidth}%
  \fi%
  \global\let\svgwidth\undefined%
  \global\let\svgscale\undefined%
  \makeatother%
  \begin{picture}(1,0.23752244)%
    \lineheight{1}%
    \setlength\tabcolsep{0pt}%
    \put(0,0){\includegraphics[width=\unitlength,page=1]{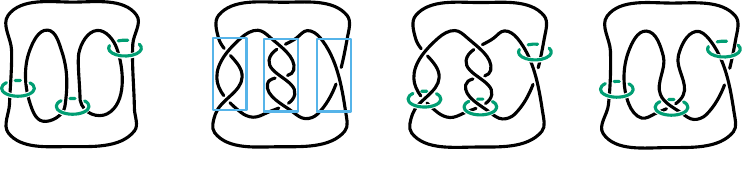}}%
    \put(0.06948597,0.00690959){\color[rgb]{0,0,0}\makebox(0,0)[lt]{\lineheight{1.25}\smash{\begin{tabular}[t]{l}(A)\end{tabular}}}}%
    \put(0.36354254,0.00590253){\color[rgb]{0,0,0}\makebox(0,0)[lt]{\lineheight{1.25}\smash{\begin{tabular}[t]{l}(B)\end{tabular}}}}%
    \put(0.6294019,0.00892367){\color[rgb]{0,0,0}\makebox(0,0)[lt]{\lineheight{1.25}\smash{\begin{tabular}[t]{l}(C)\end{tabular}}}}%
    \put(0.88519086,0.00590253){\color[rgb]{0,0,0}\makebox(0,0)[lt]{\lineheight{1.25}\smash{\begin{tabular}[t]{l}(D)\end{tabular}}}}%
  \end{picture}%
\endgroup%

  \caption{(A) A flat fully augmented link. (B) A knot with three twist regions, with twist regions marked in blue. (C) Add crossing circles to augment the knot. (D) The corresponding fully augmented link.}
  \label{Fig:FALExample}
\end{figure}

Under mild restrictions, fully augmented links are hyperbolic with very explicit hyperbolic geometry; see~\cite{Lackenby:AltVolume, FuterPurcell, Purcell:FullyAugmented, Purcell:Cusps}. For this reason, they are well-studied in hyperbolic knot theory. Our application does not require any of their geometric properties, but uses geometry to imply the existence of a triangulation.

\begin{proposition}\label{Prop:AugGivesHyperbolic}
Any non-trivial knot $K$ in $S^3$ can be obtained from some hyperbolic fully augmented link by Dehn filling the crossing circle components.
\end{proposition}

\begin{proof}
When the knot is hyperbolic, this is well known. We step through the parts of the proof that we need. 

Let $K$ be any knot or link in $S^3$. Take a diagram of $K$ and consider its \emph{twist regions}: these are regions of the diagram where two strands twist maximally, forming a collection of bigons arranged end to end. A single crossing adjacent to no bigons is also a twist region. Figure~\ref{Fig:FALExample}(B) shows a knot diagram with three twist regions indicated. A diagram is \emph{twist reduced} if, whenever a simple closed curve on the plane of projection meets the diagram in four points, adjacent to two crossings, the curve bounds a region of the plane containing a (possibly empty) twist region (see also~ \cite[Figure~6]{Purcell:FullyAugmented}). Any diagram can be modified to be twist reduced (see for example~\cite[Lemma~4]{Lackenby:AltVolume}, or~\cite[p~5]{Purcell:Cusps}). 

Augment a twist reduced diagram by inserting a simple unknot called a \emph{crossing circle} encircling each twist region, bounding a disc that punctures the diagram twice, as in \reffig{FALExample}(C). Note that there are two ways to do this if the twist region contains a single crossing; either one is suitable. Next, adjust the diagram by removing an even number of crossings in each twist region, so that each crossing circle is adjacent to either zero or one crossings, depending on whether the original twist region had an even or an odd number of crossings. See \reffig{FALExample}(D). Note that the link obtained by removing crossings in pairs has complement homeomorphic to the link with added crossing circles, via a cut--twist--reglue homeomorphism (as in \cite[Proposition~7.2]{Purcell:HypKnotTheory}).  The original link is obtained by performing $1/n_i$ Dehn filling on the $i$th crossing circle, where $n_i$ is an integer such that $2|n_i|$ crossings were removed to form the fully augmented link, and the sign is chosen to ensure the crossings have the correct sign.

The fully augmented link constructed above is hyperbolic if and only if before adding crossing circles (or \emph{augmenting}), it came from a non-splittable, prime, twist reduced diagram with at least two twist regions; see \cite[Theorem~6.1]{Purcell:Cusps}. Because we are working with knots, all diagrams are non-splittable. As noted above, we may assume our diagrams are twist reduced. Thus it remains to consider separately the case that there is only one twist region, and the case that the diagram is not prime. 

In the case of a single twist region, the knot is a $(2,q)$-torus knot, $|q|> 2$. Given the standard diagram of the $(2,q)$-torus knot, we insert $q$ crossing circles, one per crossing, meeting bigon regions of the diagram rather than encircling twist regions; see \reffig{FALT2q}, left. 
This link complement is homeomorphic to the link complement where $2n$ crossings are added adjacent to each crossing circle, $n \geq 3$. Thus the link complement is also obtained from augmenting a prime, twist reduced diagram with $|q|>2$ twist regions.\footnote{This is the augmented diagram of a pretzel link with $|q|>2$ twist regions and $2n+1$ crossings per twist region.} Then \cite[Theorem~6.1]{Purcell:Cusps} proves that the fully augmented link is hyperbolic. 
Performing $1/0$ Dehn fillings on all crossing circles in the original fully augmented link diagram returns the original $(2,q)$-torus knot.

\begin{figure}[ht]
  \includegraphics{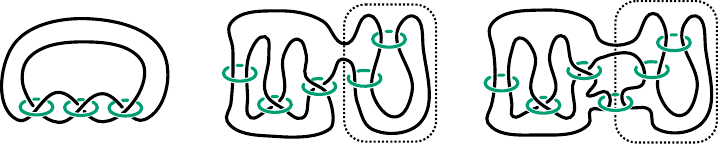}
  \caption{Left to right: augmenting a $(2,q)$-torus knot to a hyperbolic fully augmented link; an example of a fully augmented link with a diagram that is not prime; adjustment of the diagram to produce a hyperbolic fully augmented link.}
  \label{Fig:FALT2q}
\end{figure}

If the diagram is not prime, by definition there exists a simple closed curve $\gamma$ in the plane of the diagram, meeting the diagram exactly twice, bounding discs on either side in the plane of projection that contain crossings. When we augment as above, the curve $\gamma$ meets exactly two knot strands of the fully augmented link, and the discs on either side contain crossing circles. This is a connected sum of augmented links $L_1$ and $L_2$, connected along two knot strands. We show that in the case that $L_1$ and $L_2$ are hyperbolic fully augmented links we may modify the construction to obtain a new hyperbolic fully augmented link, with the original knot obtained from Dehn filling this new link. The result then follows by induction on the number of curves $\gamma$.

For the modification, the connected sum of $L_1$ and $L_2$ occurs along two knot strands bounding regions of each diagram that are merged after the connected-sum procedure. Take knot strands from each of $L_1$ and $L_2$ bounding adjacent regions of the diagram, and isotope them to run up to either side of the curve $\gamma$. Add a crossing circle between these two strands, with the crossing circle piercing $\gamma$. An example of this is shown on the right of \reffig{FALT2q}, where $\gamma$ is given as the dotted line. We claim that the resulting augmented link is a prime, twist reduced diagram. Twist reduced follows from the fact that $L_1$ and $L_2$ arise from twist reduced diagrams, and the newly added crossing circle cannot be parallel to an existing crossing circle since it straddles the two links. For primeness, the curve $\gamma$ no longer meets the diagram exactly twice. 
Suppose for contradiction that $\delta$ is a simple closed curve in the plane of the diagram meeting the link diagram exactly twice, bounding discs on either side in the plane of projection that contain crossings. No such $\delta$ exists that avoids $\gamma$, since such a $\delta$ must lie in the diagram of $L_1$ or $L_2$, which we have assumed to be hyperbolic (hence prime) and twist-reduced. But if $\delta$ runs through a neighbourhood of $\gamma$,
our construction ensures that every such curve meets the diagram more than twice, yielding the desired contradiction.
\end{proof}

\begin{lemma}\label{Lem:SubdivideFAL}
A hyperbolic fully augmented link complement can be subdivided into an ideal triangulation with at most $2 \cdot(6c-4)$ ideal tetrahedra, where $c$ is the number of crossing circles.
In this ideal triangulation, each cusp neighbourhood of a crossing circle is met by exactly four tetrahedra, meeting in one ideal vertex. 
For each crossing circle, there are four faces of the ideal triangulation that, together, form the twice-punctured disc bounded by a crossing circle. 
\end{lemma}

Note that the tetrahedron count in \reflem{SubdivideFAL} is optimised for links with a small number of twist regions relative to the overall crossing number. For links where the average number of crossings per twist region is small, there are likely modifications of our argument leading to a smaller tetrahedron count. But we do not focus on this in this paper. See \refsec{trefoils} for some discussion and clarifying remarks.

\begin{proof}[Proof of \reflem{SubdivideFAL}]
Consider the standard geometric decomposition of a hyperbolic fully augmented link into identical convex, right angled, ideal hyperbolic polyhedra. This was first explained in ~\cite[Appendix]{Lackenby:AltVolume}; see also~\cite{Purcell:FullyAugmented}. Note that in \cite[Appendix]{Lackenby:AltVolume}, the authors proceed to refine this decomposition into an ideal triangulation with $10(c-1)$ tetrahedra, which is different from the ideal triangulation we present here.

In short, the decomposition is illustrated in \reffig{FALDecomp}, which first appeared in ~\cite{HamPurcell}. First slice through each 2-punctured disc bounded by a crossing circle, and then rotate by $180^\circ$ where necessary to untwist any single crossings of knot strands. Next cut along the plane of projection, separating the link complement into two identical pieces. Then shrink remnants of the link to ideal vertices. The resulting ideal polyhedra admit a complete convex hyperbolic structure, with faces coming from the projection plane corresponding to a circle packing, and the 2-punctured discs bounded by crossing circles corresponding to a dual circle pattern.
\begin{figure}
  \includegraphics{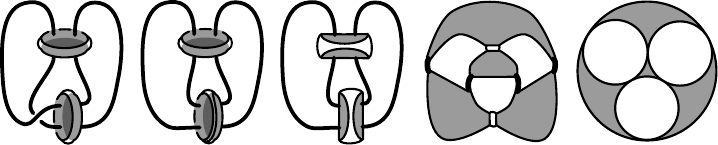}
  \caption{The decomposition of a hyperbolic fully augmented link into polyhedra, from~\cite{HamPurcell}.}
  \label{Fig:FALDecomp}
\end{figure}

Observe that four shaded ideal triangles make up each crossing disc in the decomposition: two from each of the two identical polyhedra, with each pair meeting in ``bow ties.'' These are the shaded triangles in \reffig{FALDecomp}. 

We subdivide these polyhedra as in \cite[Proposition~2.3]{HamPurcell} without affecting the shaded triangles. We briefly step through the details to confirm the required properties.

Choose a crossing circle. This circle corresponds to an ideal vertex in each of the two identical convex hyperbolic polyhedra making up the decomposition. Denote these polyhedra by $P_1$ and $P_2$. We first focus on one of the polyhedra, say, $P_1$. Take the ideal vertex corresponding to a crossing circle to infinity in the hyperbolic structure. The two adjacent shaded triangles lift to lie on parallel vertical hyperbolic planes, which lie over parallel Euclidean lines on the boundary at infinity $\bdy_\infty\HH^3$. The two adjacent white faces, coming from faces of the plane of projection, lift to vertical planes lying over parallel Euclidean lines that meet the shaded faces orthogonally, forming a rectangle on $\bdy_\infty\HH^3$ as in \reffig{CrossingDiscsFlat}(B).

\begin{figure}
\begingroup%
  \makeatletter%
  \providecommand\color[2][]{%
    \errmessage{(Inkscape) Color is used for the text in Inkscape, but the package 'color.sty' is not loaded}%
    \renewcommand\color[2][]{}%
  }%
  \providecommand\transparent[1]{%
    \errmessage{(Inkscape) Transparency is used (non-zero) for the text in Inkscape, but the package 'transparent.sty' is not loaded}%
    \renewcommand\transparent[1]{}%
  }%
  \providecommand\rotatebox[2]{#2}%
  \newcommand*\fsize{\dimexpr\f@size pt\relax}%
  \newcommand*\lineheight[1]{\fontsize{\fsize}{#1\fsize}\selectfont}%
  \ifx\svgwidth\undefined%
    \setlength{\unitlength}{257.87768555bp}%
    \ifx\svgscale\undefined%
      \relax%
    \else%
      \setlength{\unitlength}{\unitlength * \real{\svgscale}}%
    \fi%
  \else%
    \setlength{\unitlength}{\svgwidth}%
  \fi%
  \global\let\svgwidth\undefined%
  \global\let\svgscale\undefined%
  \makeatother%
  \begin{picture}(1,0.33450628)%
    \lineheight{1}%
    \setlength\tabcolsep{0pt}%
    \put(0,0){\includegraphics[width=\unitlength,page=1]{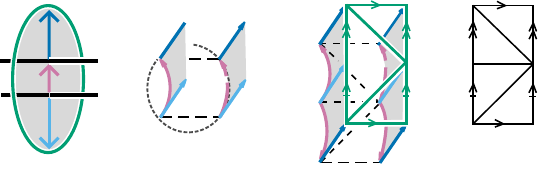}}%
    \put(0.00984524,0.01059742){\color[rgb]{0,0,0}\makebox(0,0)[lt]{\lineheight{1.25}\smash{\begin{tabular}[t]{l}(A)\end{tabular}}}}%
    \put(0.27074449,0.0096129){\color[rgb]{0,0,0}\makebox(0,0)[lt]{\lineheight{1.25}\smash{\begin{tabular}[t]{l}(B)\end{tabular}}}}%
    \put(0.51920666,0.00764386){\color[rgb]{0,0,0}\makebox(0,0)[lt]{\lineheight{1.25}\smash{\begin{tabular}[t]{l}(C)\end{tabular}}}}%
    \put(0.8781967,0.01059742){\color[rgb]{0,0,0}\makebox(0,0)[lt]{\lineheight{1.25}\smash{\begin{tabular}[t]{l}(D)\end{tabular}}}}%
    \put(0,0){\includegraphics[width=\unitlength,page=2]{CrossingDisc_Flat1.pdf}}%
  \end{picture}%
\endgroup%

  \caption{(A) Three important ideal edges on a crossing disc. (B) Lifting the ideal vertex corresponding to a crossing circle to infinity in the upper half space model, viewed from $\infty$, with the green rectangle closest to the viewer and highest in upper half space. 
	The dashed circle  indicates a hyperbolic plane that we use to cut off two triangles. (C) A fundamental region with choice of triangulation of this cusp. (D) A cusp triangulation.}
  \label{Fig:CrossingDiscsFlat}
\end{figure}

Because the vertical white and shaded faces meet the crossing circle cusp in a rectangle, there is a Euclidean circle meeting each of its corners, slicing off a geodesic plane in $\HH^3$. This slices $P_1$ into two parts: a rectangular pyramid with base the geodesic plane, giving a cusp neighbourhood of the crossing circle ideal vertex, and a (possibly degenerate) convex ideal polyhedron in its complement, as in~\cite[Lemma~2.2]{HamPurcell}. Triangulate the cusp neighbourhood by choosing a diagonal. 

Now consider $P_2$. A face pairing of a white side glues this polyhedron to the original, giving a fundamental region for the link complement. Slice off a geodesic plane similarly,
and triangulate the cusp by reflecting the diagonal, as shown in \reffig{CrossingDiscsFlat}(C). This gives the cusp triangulation as claimed.

This completes the construction of the crossing circle cusp triangulation in the flat case. In the case that the crossing circle meets a crossing, or half-twist, the cusp shape is sheared under the face pairing identification, as in \reffig{CrossingDiscHalfTwist}. The cusp is still triangulated as claimed.
\begin{figure}
  \includegraphics{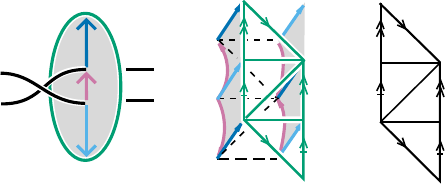}
  \caption{When the crossing disc meets a half twist, the fundamental region is the same, but the face pairing shears the cusp torus. The cusp triangulation is shown on the right.}
  \label{Fig:CrossingDiscHalfTwist}
\end{figure}

It remains to subdivide the remainder of the polyhedra, and to count the number of tetrahedra. First note that, by \cite[Lemma~2.2]{HamPurcell}, the remnant may be degenerate. In this case, the original fully augmented link had only two crossing circles, hence there are $8 = 2\cdot 4$ tetrahedra in total. See \cite[Figure~3]{HamPurcell} and \cite[Figure~ 15]{Lackenby:AltVolume} for two informative illustrations of this situation.

Otherwise, after inserting the two ideal tetrahedra of $P_1$ incident to a given crossing circle cusp, the corresponding bow tie of shaded faces, together with triangular portions of the adjacent white faces thickening the bow tie into a rectangle, are covered by two triangles triangulating this rectangle. In particular, after covering all $c$ ideal vertices corresponding to crossing circles this way, the modified version of $P_1$, as seen from its centre, now has $2c$ remaining ideal vertices and, after triangulating all of its faces, $4c-4$ triangles. Reflecting through the white faces shows the same result for $P_2$. Denote the resulting polyhedra by $P_1'$ and $P_2'$.

We now subdivide each of $P_1'$ and $P_2'$ into $4c-4$ tetrahedra by coning over their centres. This yields an overall decomposition into $2\cdot (4c-4)+4c = 2\cdot (6c-4)$ tetrahedra, as desired,
although we have added new material vertices.  
To obtain an ideal triangulation, instead cone over an existing vertex. In fact, this yields an improvement of at least $8$ tetrahedra in the non-degenerate case. 
However, the bound proven here is the right order of magnitude.
\end{proof}

\begin{lemma}\label{Lem:PrismFAL}
The ideal triangulation of \reflem{SubdivideFAL} can be adjusted such that 
\begin{enumerate}
\item for each crossing circle, there is a triangular face spanning a meridian of the link;
\item each crossing circle cusp meets exactly two tetrahedra in exactly one ideal vertex.
\end{enumerate}
The adjustment replaces the four tetrahedra adjacent to a crossing circle by five tetrahedra if the crossing circle meets no crossing, and four tetrahedra if it is adjacent to a crossing.
\end{lemma}

\begin{proof}
Remove the four tetrahedra adjacent to a crossing circle. This leaves an object with 2-punctured torus boundary, with decomposition into tetrahedra as in Figures~\ref{Fig:CrossingDiscsFlat} and~\ref{Fig:CrossingDiscHalfTwist}, with the pink edges identified.

When there is no half twist, the 2-punctured torus is given a decomposition into triangles as in \reffig{CrossingDiscsFlat}(D). 
Attach to this a triangular prism as in \reffig{Prism}(A) and (B). This has three rectangular faces, two of which match the two rectangles of the 2-punctured torus, and the remaining face capping off these two rectangles. The other two faces of the prism are triangles, glued to the ideal shaded faces shown in \reffig{CrossingDiscsFlat}(C).

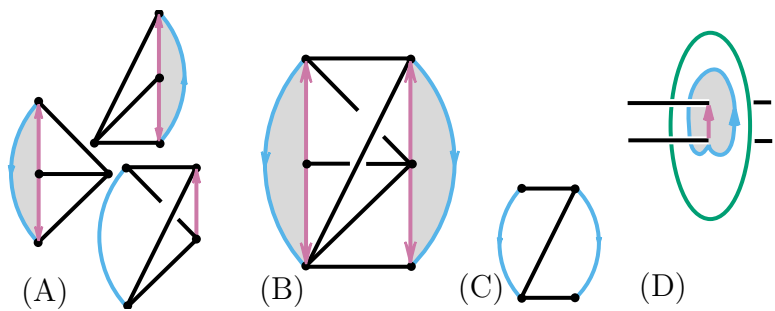
\begin{figure}[htb]
\centering
\begingroup%
  \makeatletter%
  \providecommand\color[2][]{%
    \errmessage{(Inkscape) Color is used for the text in Inkscape, but the package 'color.sty' is not loaded}%
    \renewcommand\color[2][]{}%
  }%
  \providecommand\transparent[1]{%
    \errmessage{(Inkscape) Transparency is used (non-zero) for the text in Inkscape, but the package 'transparent.sty' is not loaded}%
    \renewcommand\transparent[1]{}%
  }%
  \providecommand\rotatebox[2]{#2}%
  \newcommand*\fsize{\dimexpr\f@size pt\relax}%
  \newcommand*\lineheight[1]{\fontsize{\fsize}{#1\fsize}\selectfont}%
  \ifx\svgwidth\undefined%
    \setlength{\unitlength}{288.06313705bp}%
    \ifx\svgscale\undefined%
      \relax%
    \else%
      \setlength{\unitlength}{\unitlength * \real{\svgscale}}%
    \fi%
  \else%
    \setlength{\unitlength}{\svgwidth}%
  \fi%
  \global\let\svgwidth\undefined%
  \global\let\svgscale\undefined%
  \makeatother%
  \begin{picture}(1,0.38898321)%
    \lineheight{1}%
    \setlength\tabcolsep{0pt}%
    \put(0,0){\includegraphics[width=\unitlength,page=1]{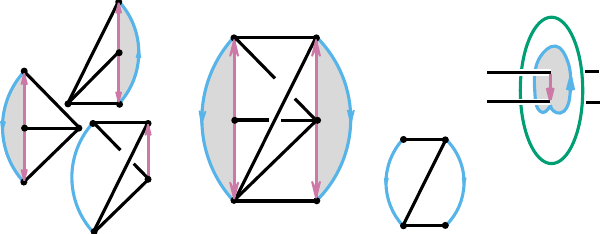}}%
    \put(0.01748969,0.00881252){\color[rgb]{0,0,0}\makebox(0,0)[lt]{\lineheight{1.25}\smash{\begin{tabular}[t]{l}(A)\end{tabular}}}}%
    \put(0.32862212,0.01157404){\color[rgb]{0,0,0}\makebox(0,0)[lt]{\lineheight{1.25}\smash{\begin{tabular}[t]{l}(B)\end{tabular}}}}%
    \put(0.58886476,0.0143356){\color[rgb]{0,0,0}\makebox(0,0)[lt]{\lineheight{1.25}\smash{\begin{tabular}[t]{l}(C)\end{tabular}}}}%
    \put(0.82569759,0.01525607){\color[rgb]{0,0,0}\makebox(0,0)[lt]{\lineheight{1.25}\smash{\begin{tabular}[t]{l}(D)\end{tabular}}}}%
  \end{picture}%
\endgroup%

  \caption{(A) The three tetrahedra making up a triangulated prism, (B)~the tetrahedra stacked into the triangulated prism, (C) the two triangles facing the crossing circle cusp, (D) the edges corresponding to blue and pink edges shown on the crossing disc in the fully augmented link.}
  \label{Fig:Prism}
\end{figure}

The shaded triangular face on the left of that prism (as well as the one on the right) now spans a meridian. Recall that the two pink edges of the face are both identified to the edge running between punctures on the 2-punctured crossing disc bounded by the crossing circle. The third edge (light blue) of the triangle runs from one of the punctures of the crossing circle disc to encircle a meridian of the other knot strand component.

In the case that there is a crossing adjacent to a crossing circle, the 2-punctured torus, which is obtained by removing the four tetrahedra meeting the crossing circle cusp, has a decomposition into triangles as in \reffig{Prism_other}. Here, the situation is simpler. Layer a tetrahedron over the horizontal edge in the centre, covering the two triangles sharing an edge with the left boundary of the fundamental domain. One of the remaining triangles of the layered tetrahedra is now a shaded face spanning a meridian, and the remaining triangular face is now facing the cusp. Attach a similar tetrahedron to the remaining two triangles sharing an edge with the right boundary of the fundamental domain. After identifying the shaded faces, two triangles are left over facing the crossing circle cusp, see \reffig{Prism_other} for an illustration.

\begin{figure}[htb]
\centering
\begingroup%
  \makeatletter%
  \providecommand\color[2][]{%
    \errmessage{(Inkscape) Color is used for the text in Inkscape, but the package 'color.sty' is not loaded}%
    \renewcommand\color[2][]{}%
  }%
  \providecommand\transparent[1]{%
    \errmessage{(Inkscape) Transparency is used (non-zero) for the text in Inkscape, but the package 'transparent.sty' is not loaded}%
    \renewcommand\transparent[1]{}%
  }%
  \providecommand\rotatebox[2]{#2}%
  \newcommand*\fsize{\dimexpr\f@size pt\relax}%
  \newcommand*\lineheight[1]{\fontsize{\fsize}{#1\fsize}\selectfont}%
  \ifx\svgwidth\undefined%
    \setlength{\unitlength}{306.06324005bp}%
    \ifx\svgscale\undefined%
      \relax%
    \else%
      \setlength{\unitlength}{\unitlength * \real{\svgscale}}%
    \fi%
  \else%
    \setlength{\unitlength}{\svgwidth}%
  \fi%
  \global\let\svgwidth\undefined%
  \global\let\svgscale\undefined%
  \makeatother%
  \begin{picture}(1,0.39211309)%
    \lineheight{1}%
    \setlength\tabcolsep{0pt}%
    \put(0,0){\includegraphics[width=\unitlength,page=1]{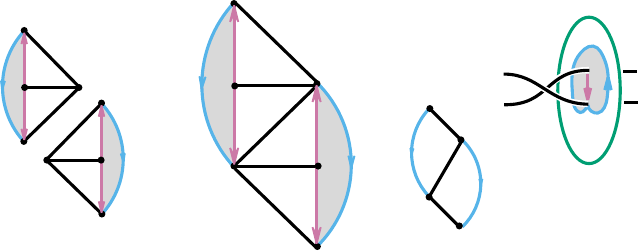}}%
    \put(0.01646109,0.03430089){\color[rgb]{0,0,0}\makebox(0,0)[lt]{\lineheight{1.25}\smash{\begin{tabular}[t]{l}(A)\end{tabular}}}}%
    \put(0.30929529,0.04180096){\color[rgb]{0,0,0}\makebox(0,0)[lt]{\lineheight{1.25}\smash{\begin{tabular}[t]{l}(B)\end{tabular}}}}%
    \put(0.60324215,0.03949915){\color[rgb]{0,0,0}\makebox(0,0)[lt]{\lineheight{1.25}\smash{\begin{tabular}[t]{l}(C)\end{tabular}}}}%
    \put(0.82614642,0.04036548){\color[rgb]{0,0,0}\makebox(0,0)[lt]{\lineheight{1.25}\smash{\begin{tabular}[t]{l}(D)\end{tabular}}}}%
  \end{picture}%
\endgroup%

  \caption{(A) Two tetrahedra to layer onto the two-punctured torus (B)~as shown on the cusp neighbourhood, (C) as a result we have two triangles facing the crossing circle cusp. (D) The edges corresponding to blue and pink edges shown on the crossing disc in the fully augmented link.}
  \label{Fig:Prism_other}
\end{figure}
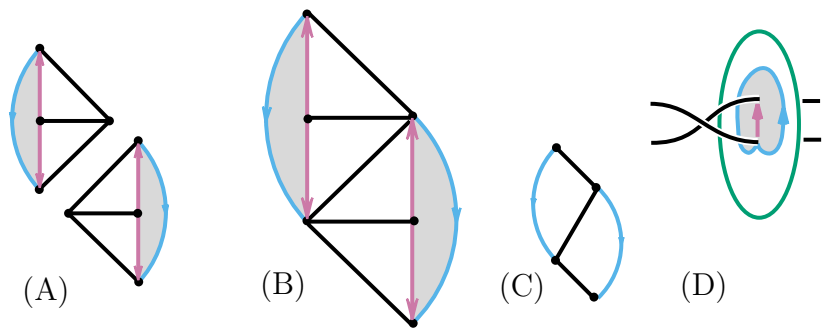

Finally, in both cases, attach two ideal tetrahedra, each with one face on one of the two ideal triangles forming the new boundary after attaching the triangulated prism. The opposite ideal vertex is mapped to the crossing circle cusp, forming the cusp triangulation as claimed.
Repeat this for all crossing circles, replacing four tetrahedra per crossing circle with four or five tetrahedra. 
\end{proof}

\begin{lemma}\label{Lem:LST}
Suppose a cusp meets exactly two ideal tetrahedra in exactly one ideal vertex. Then a $1/n$ Dehn filling along that cusp can be realised by removing these two ideal tetrahedra, and adding a solid torus with an ideal triangulation containing at most $|n|-1$ tetrahedra if $|n|\geq 2$, and at most one tetrahedron if $|n|=1$ or $n=0$. 
\end{lemma}

\begin{proof}
If $|n|>2$, the Dehn filling is obtained by inserting a layered solid torus, as in \cite{GueritaudSchleimer}; see also \cite[Section~3.1]{HamPurcell}, or \cite[Section~3.1]{HowieMathewsPurcell}. 
After removing the original two tetrahedra, we have an object with boundary a 1-punctured torus decomposed into two triangles. Such decompositions of 1-punctured tori are organised by the Farey complex of $\HH^2$. Namely, they are determined by isotopy classes of three curves on the torus that mutually intersect exactly once (called \emph{slopes}). These slopes each correspond to a rational number, and the three define a triangle in the Farey complex.  The solid torus that gives the Dehn filling is obtained by walking through the Farey complex: Start at the Farey triangle  corresponding to the three initial slopes; finish exactly one step away from a triangle in the Farey complex meeting the slope corresponding to the Dehn filling, and then fold unglued faces. When the original triangles on the 1-punctured torus have edges coming from slopes labelled $0/1$, $1/0$, and $\pm 1/1$, as is the case for the triangulations for \reflem{PrismFAL}, there are at most $|n|$ steps in the Farey graph between the Farey triangle of these slopes and a Farey triangle meeting slope $1/n$. Hence we attach at most $|n|-1$ tetrahedra to perform the Dehn filling. 

When $|n|=2$, the distance between the original Farey triangle and the Dehn filling slope is either $1$ or $2$, so we either attach one tetrahedron and then fold, or we may fold immediately. In either case, at most $|n|-1 = 1$ tetrahedron is added.

When $|n|=1$ ($n=0$), the slope of the Dehn filling may be (is) one of the existing edges of the triangulation. In this case, layer on a tetrahedron covering that edge, then fold so that edge bounds a disc. This was called the \emph{extra-exceptional case} in \cite{HowieMathewsPurcell}. 
If the Dehn filling slope is not one of the existing edges, the distance between the original triangle and the Dehn filling slope is $1$, and we can fold without layering a tetrahedron. In any case, the result is a triangulation of the Dehn filling.
\end{proof}

\begin{proposition}\label{Prop:FullyAugLinks}
Every knot $K$ obtained by Dehn filling crossing circles in a hyperbolic fully augmented link in $S^3$ has complement $S^3-K$ admitting an ideal triangulation in which a face spans a meridian. 
\end{proposition}

\begin{proof}
By \reflem{SubdivideFAL}, we can subdivide a fully augmented link into tetrahedra with faces forming an ideal triangulation of the disc bounding a crossing circle. By \reflem{PrismFAL}, this ideal triangulation can be adjusted to obtain a face spanning a meridian for each crossing circle. By \reflem{LST}, we can then perform Dehn fillings by replacing two tetrahedra per crossing circle with a solid torus with an ideal triangulation. Observe that no other tetrahedron in the decomposition is changed by the Dehn filling. In particular a face spanning the meridian of the knot strand remains a face spanning the meridian. Repeat this for each crossing circle. 
The final result is an ideal triangulation of the knot complement with a face spanning a meridian for each crossing circle. This satisfies the conclusion of the proposition. 
\end{proof}

\begin{theorem}\label{Thm:1vertex}
\Paste{Statement:1vertex}
\end{theorem}

\begin{proof}
Let $K$ be a knot in $S^3$. By \refprop{AugGivesHyperbolic}, $K$ can be obtained from some hyperbolic fully augmented link by Dehn filling the crossing circle components. By \refprop{FullyAugLinks}, $S^3-K$ admits an ideal triangulation in which a face (a hat triangle) spans a meridian. \refcor{KnottedEdgeS3} proves the result; a one-vertex triangulation of $S^3$ with an edge forming $K$ is obtained from $\tri$ by inserting the folded tetrahedron $\pmb \Delta$ along a hat triangle.
\end{proof}

\section{Applications}\label{Sec:Bounds}

In this section we present two applications of our construction. We start by summarising the results from \refsec{mertriangle} regarding the size of our triangulations.

\begin{theorem}\label{Thm:TetrahedronCount}
Suppose $K$ is obtained by Dehn filling a hyperbolic fully augmented link with $c$ crossing circles, along Dehn filling slopes $1/n_1, \dots 1/n_c$. Then the one-vertex triangulation from \refthm{1vertex} contains at most $12c + \sum |n_i| - 7$ tetrahedra.
\end{theorem}

\begin{proof}
We first find an ideal triangulation of the fully augmented link as explained in the proof of \reflem{SubdivideFAL}. This yields an ideal triangulation with $2\cdot(6c-4) = 12c -8$ tetrahedra. Adjusting the ideal triangulation to obtain two-triangle crossing circle cusps following \reflem{PrismFAL} adds at most one tetrahedron per crossing circle cusp, bringing the total to at most $13c-8$ tetrahedra. By \reflem{LST}, every Dehn filling $1/n_i$, with $|n_i| > 1$, removes two tetrahedra and replaces them by at most $|n_i|-1$ tetrahedra, while every Dehn filling along slope $1/(\pm 1)$ or $1/0$ removes two tetrahedra and adds at most a single extra tetrahedron.

It follows that the overall number of tetrahedra is bounded above by $13c-8 - 2c + c + \sum |n_i| = 12c-8 + \sum |n_i|$.
Finally, adding the single tetrahedron $\Delta$ via \refcor{KnottedEdgeS3} yields the result.
\end{proof}

\begin{corollary}
\label{Cor:hyperbolic_knot_tet_count}
Let $K$ be a hyperbolic knot in $S^3$. Suppose $K$ has a prime, twist-reduced diagram $D$ with $c$ twist regions, containing $t_1, \ldots, t_c$ crossings respectively. Then there is a one-vertex triangulation of $S^3$ with an edge forming $K$, and at most $12c+ \sum_{j=1}^c \lfloor t_j/2 \rfloor - 7$ tetrahedra.
\end{corollary}

\begin{proof}
Following the standard procedure as in \refprop{AugGivesHyperbolic}, augment $D$ with crossing circles around each of the $c$ twist regions. Remove an even number of crossings from each twist region, to obtain a fully augmented link $L$. Then $S^3 \setminus K$ is obtained from $S^3 \setminus L$ by $1/n_i$ Dehn filling the $i$th crossing circle, where $2|n_i| = t_i$ or $t_i - 1$ accordingly as $t_i$ is even or odd, i.e.\ $|n_i| = \lfloor t_i/2\rfloor$. As $D$ is prime, twist reduced, and must have at least two twist regions since $K$ is hyperbolic, $L$ is hyperbolic. \refthm{TetrahedronCount} then gives the result.
\end{proof}

A straightforward consequence of \refthm{TetrahedronCount} is that our construction is most efficient for prime knots with few twist regions or, more generally, for knots with few crossing circles (relative to their numbers of crossings) in their fully augmented link diagrams. In what follows, we discuss two families of examples that are extreme with respect to this observation: one with only two crossing circles (\refsec{twistknots}), and one with the number of crossing circles exceeding the crossing number of the initial diagram (\refsec{trefoils}).

\subsection{Knots with few crossing circles}
\label{Sec:twistknots}

Consider a twist-reduced, crossing minimal diagram of the double twist knot $J(k,\ell)$, with $k$ and $\ell$ crossings, $k,\ell \geq 3$, in its two respective twist regions. Its fully augmented link diagram has two crossing circles and either $0$ or $1$ crossings depending on the parities of $k$ and $\ell$. (The case when $k$ and $\ell$ are both odd is omitted because in this case $J(k, \ell)$ is a two-component link.) The knot $J(k,\ell)$ is recovered from this diagram by performing $1/\lfloor k/2 \rfloor$ and $1/\lfloor \ell/2 \rfloor$ Dehn fillings on the respective crossing circles.

Following \cite[Lemma~2.2]{HamPurcell}, we can give an ideal triangulation of the fully augmented link complement starting from the $8$-tetrahedra ideal triangulation with degenerate remnant, as described in \reflem{SubdivideFAL}. We then have to distinguish two cases when Dehn filling: an even and an odd number of twists in a twist region.

If we have an even number $n>2$ of twists in a twist region, the respective crossing circle is not adjacent to a crossing. Inserting the prism construction from \reffig{Prism} adds three tetrahedra. Note that we get to choose the diagonal of the two boundary triangles facing the cusp and we hence only need to layer $n/2-2$ more tetrahedra before we fold to realise a $1/(n/2)$ Dehn filling. Altogether, this requires $1+ n/2$ tetrahedra.

If we have an odd number $n>1$ of twists in a twist region, the respective crossing circle is adjacent to a crossing. We first layer two tetrahedra as shown in \reffig{Prism_other}. Note that this time, we do not get to choose the diagonal which, moreover, is always facing in the wrong direction, causing us to layer another $(n-1)/2-1$ tetrahedra before we can fold. Altogether, this requires $1+(n-1)/2$ tetrahedra.

It follows that the complement of $J(k,\ell)$ has an ideal triangulation with a meridional hat triangle and $2 + \lfloor k/2 \rfloor + \lfloor \ell/2 \rfloor$ tetrahedra, and hence there exists a one-vertex triangulation of $S^3$ with $3 + \lfloor k/2 \rfloor + \lfloor \ell/2 \rfloor$ tetrahedra and an edge forming $J(k,\ell)$.

The ideal triangulations are conjectured to use the smallest possible number of tetrahedra in work by the first three authors of this article \cite{Ibarra24preprint}.

\subsection{Knots with many crossing circles}
\label{Sec:trefoils}

As mentioned above, the $\Delta$-complex constructed to prove \refthm{1vertex} has some inefficiencies when the number of crossing circles is high with respect to the number of crossings or, equivalently, the Dehn fillings have small slopes, such as $1/0$ or $1/1$. However, even in this case our construction yields interesting results. We focus here on the $t$-fold connected sum of the trefoil $K_t$.
\begin{corollary}\label{Cor:trefoilsphere}
Let $K_t$ be a $t$-fold connected sum of the trefoil knot. Then there exists a one-vertex triangulation of the $3$-sphere with an edge forming $K_t$ and at most $48t - 19$ tetrahedra.
\end{corollary}

\begin{proof}
The knot $K_t$ famously has crossing number $3t$, and, following our construction in \refsec{mertriangle}, a fully augmented link diagram of $K_t$ requires $4t-1$ crossing circles: one per crossing, and one crossing circle per connected sum operation. See \reffig{figureoftrefoils} for an illustration of such a fully augmented link diagram. Moreover, to recover $K_t$ from this fully augmented link diagram, each crossing circle must be Dehn-filled with slope $1/0$. The result now directly follows from the (non-optimal) bound given in \refthm{TetrahedronCount}.  
\end{proof}
 
\begin{figure}[ht]
\centering
  \includegraphics{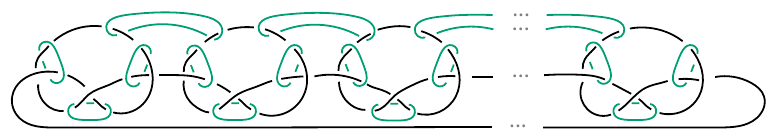}
  \caption{Fully augmented link diagram of a $t$-fold connected sum of trefoils.}
  \label{Fig:figureoftrefoils}
\end{figure}

Passing to the second derived subdivision of the triangulation from \refcor{trefoilsphere} immediately yields the following result.

\begin{corollary}\label{Cor:trefoilsimplicialsphere}
There exists a (simplicial) triangulation of the $3$-sphere with at most $24^2 \cdot (48t - 19)$ tetrahedra, containing an edge loop of length $4$ forming a $t$-fold connected sum of the trefoil.
\end{corollary}

\refcor{trefoilsimplicialsphere} combined with \refthm{benedettilickorish} now finishes the proof of \refcor{DiscreteMorseCor}, that there exists a simplicial triangulation of the $3$-sphere $\tri$ with $24^2 \cdot (48t - 19)$ tetrahedra such that every discrete Morse function on $\tri$ must have at least $t-3$ critical $2$-faces. We use \refthm{benedettilickorish} in the opposite direction to prove a lower bound on the size of a triangulation of the $3$-sphere with a short edge loop forming $K_t$.

\begin{corollary}\label{Cor:LowerBound}
Let $K_t$ denote the $t$-fold connected sum of the trefoil. 
For any positive integer $t$, and any positive integer $m$, suppose a (simplicial) triangulation of the $3$-sphere contains an edge loop of length at most $m$ forming $K_t$. Then the triangulation must contain at least $\frac12(t-m+1)$ tetrahedra.
\end{corollary}

\begin{proof}
Fix positive integers $t$ and $m$. Take a (simplicial) triangulation of the $3$-sphere $\tri$ with an edge loop of length at most $m$ forming $K_t$, and suppose the triangulation has $R$ tetrahedra. 

By \refthm{benedettilickorish}, every discrete Morse function on $\tri$ must have at least $t-m+1$ critical triangles, and hence must contain at least $t-m+1$ triangles overall. On the other hand, $\tri$ has $R$ tetrahedra and hence $2R$ triangles. 
\end{proof}

In particular, when $m=4$, such a (simplicial) triangulation has $\Omega(t)$ tetrahedra, showing that the triangulation of \refcor{trefoilsimplicialsphere} is asymptotically optimal up to a constant factor.

\begin{remark}
A technique in work by the fourth author together with He and Sedgwick~\cite{He24preprint}, contains an alternative procedure to obtain a slightly smaller one-vertex  triangulation of the $3$-sphere with an edge forming a $t$-fold connected sum of the trefoil $K_t$. This procedure triangulates crossing gadgets using $9$ tetrahedra each. Starting with a diagram of $K_t$ with $3t$ crossings, these are assembled to form a $27t$-tetrahedron $\Delta$-complex for the $3$-sphere with an edge-loop of length $6t$ forming $K_t$. Using an improved version of Jeff Weeks' tunneling trick, this can be adjusted to a $(33t-1)$-tetrahedron one-vertex 
triangulation of the $3$-sphere containing $K_t$ as a single loop-edge. 

\medskip

The same work also contains a lower bound of such a $\Delta$-complex of
$2t$, $t \geq 2$: There exists a $1$-tetrahedron $\Delta$-complex for
the $3$-sphere containing an edge loop forming the trefoil. A $t$-fold connected sum of the trefoil is then shown to contain at least $t + \sum \limits_{i=1}^{t} 1  = 2t$ tetrahedra. This bound is not expected to be tight, but provides an alternative proof that our $\Delta$-complexes are optimal up to a constant multiplicative factor.
\end{remark}

\section{The algorithm}
\label{Sec:Algorithm}

The proof of \refthm{1vertex} is constructive. In this section, we summarise the steps in an algorithm to create the one-vertex triangulation.

\begin{enumerate}
\item Given a knot $K$ in $S^3$, build a fully augmented link $L$ such that $K$ in $S^3$ is obtained by Dehn filling $S^3-L$ (\refprop{AugGivesHyperbolic}). 
  \begin{enumerate}
  \item Augment a twist reduced diagram of $K$ with crossing circles in each twist region, and remove even numbers of crossings so that each crossing circle is adjacent to one or zero crossings. 
  \item If the diagram is prime with at least two twist regions, it describes a hyperbolic link and we are done.
  \item If the diagram is prime with just one twist region, it describes a $(2,q)$-torus knot. Insert crossing circles for each crossing as in \reffig{FALT2q}, left.
  \item If the diagram is not prime, insert crossing circles at regions where the diagram fails to be prime, as in \reffig{FALT2q}, right.
  \end{enumerate}

\item Given a hyperbolic fully augmented link, decompose into ideal tetrahedra with exactly four ideal tetrahedra meeting each crossing circle (\reflem{SubdivideFAL}).
  \begin{enumerate}
  \item Subdivide into two ideal polyhedra following the standard geometric decomposition, for example as in~\cite{Purcell:FullyAugmented}.
  \item Subdivide the two ideal polyhedra into ideal tetrahedra following the procedure as in \reflem{SubdivideFAL}, restated from~\cite[Proposition~2.3]{HamPurcell}. 
  \end{enumerate}

\item Adjust the ideal triangulation of the previous step to obtain an ideal triangulation with a triangular face spanning a meridian, and exactly two ideal tetrahedra meeting each crossing circle (\reflem{PrismFAL}).
  \begin{enumerate}
  \item As in the proof of \reflem{PrismFAL}, remove tetrahedra adjacent to a crossing circle and replace by a triangular prism, as in Figures~\ref{Fig:Prism} and~\ref{Fig:Prism_other}. In both cases, the result has a face spanning a meridian. 
  \end{enumerate}

\item For each crossing circle, insert a layered solid torus to perform the required Dehn filling to obtain $S^3-K$, now with a face spanning a meridian (\reflem{LST}).
  \begin{enumerate}
  \item At any crossing circle, the Dehn filling slope determines a walk through the Farey complex, which determines an ideal triangulation of a layered solid torus.
  \item Remove the two tetrahedra meeting the crossing circle and replace by the given layered solid torus. 
  \end{enumerate}

\item We now have an ideal triangulation of $S^3-K$ with a face $T$ spanning a meridian. Cut along $T$ and insert the folded tetrahedron $\pmb \Delta$ of \reffig{Tetrahedron} to give a one-vertex triangulation (\refthm{IdealToFiniteTriangulation}). We are done by \refcor{KnottedEdgeS3}.

\end{enumerate}

\bibliographystyle{amsplain}
\bibliography{biblio}

\end{document}